\documentclass[11pt]{article}

\usepackage[top=1in, bottom=1in, outer=1in, inner=1in, heightrounded, marginparwidth=4cm, marginparsep=.3cm]{geometry}
\makeatother
\usepackage[dvipsnames]{xcolor}
\definecolor{DarkGolden}{HTML}{b7850b}
\definecolor{UTSAb}{HTML}{1364b0}

\usepackage[colorlinks=true, linkcolor=orange!70!red, citecolor=ForestGreen, urlcolor=RoyalBlue,draft=false]{hyperref}
\usepackage{amsfonts, amsmath, amsthm, enumerate, fancyhdr,bbm,lastpage, authblk,mathtools } 

\usepackage{sectsty} 
\allsectionsfont{\color{NavyBlue}\bsifamily}

\usepackage{lettrine}
\input Zallman.fd

\usepackage{lmodern}
\usepackage[T1]{fontenc}
\usepackage[scale=.7]{miama}
 
 \usepackage{calligra}

 \usepackage{pbsi}
 \usepackage{aurical}

 \newcommand{\tbsi}[1]{{\bsifamily #1}}
 \newcommand{\tac}[1]{{\Fontauri\slshape#1}}
 \newcommand{\tdh}[1]{{\fontfamily{cmdh}\selectfont#1}}

\makeatletter 
\def\@cite#1#2{[{\color{RoyalBlue}\sffamily\bfseries#1}\if@tempswa , \tdh{#2} \fi]}
\def\@biblabel#1{[\textbf{#1}]}
\def\tagform@#1{\maketag@@@{\tdh{(\ignorespaces#1\unskip\@@italiccorr)}}}
\makeatother
\renewcommand{\eqref}[1]{\tdh{(\ref{#1})}}

\usepackage{collectbox}



\numberwithin{equation}{section}
\newtheoremstyle{RoyalBlue}{}{}{\itshape\color{RoyalBlue}}{}{\color{RoyalBlue}\sffamily\bfseries}{.}{ }{}
\newtheoremstyle{RoyalBlue_def}{}{}{\color{RoyalBlue!80!black}}{}{\color{RoyalBlue}\bfseries}{.}{ }{}
\theoremstyle{RoyalBlue}

\newtheorem{Thm}{Theorem}[section]
\newtheorem{Lem}{Lemma}[section]
\newtheorem{Prop}{Proposition}[section]

\theoremstyle{RoyalBlue_def}

\newtheorem{Rmk}{Remark}[section]

\def\ep{\varepsilon}
\def\f{\frac}

\def\grad{\nabla}

\newcommand{\abs}[1]{ \left| #1 \right|}
\newcommand{\norm}[1]{\lVert#1\rVert}

\newcommand{\inner}[2]{\left\langle #1, #2 \right\rangle}

\newcommand{\Rmnum}[1]{ \uppercase\expandafter{\romannumeral  #1}}
\newcommand{\mc}[1]{\mathcal{#1}}
\newcommand{\ubrace}[2]{\begingroup
      \color{RoyalBlue}
      \underbrace{\color{black}#1}_{#2}
      \endgroup
      }
\newcommand{\vl} {\boldsymbol{\ell}}

\newcommand{\vxi} {\boldsymbol{\xi}}
\newcommand{\vtheta} {\boldsymbol{\theta}}
\newcommand{\vchi} {\boldsymbol{\chi}}

\newcommand{\vu} {\mathbf{u}}

\newcommand{\vq} {\mathbf{q}}
\newcommand{\vp} {\mathbf{p}}
\newcommand{\vP} {\mathbf{P}}
\newcommand{\vQ} {\mathbf{Q}}

\newcommand{\vz} {\mathbf{z}}

\newcommand{\vo} {\mathbf{0}}

\newcommand{\vx} {\mathbf{x}}

\newcommand{\vphi} {\boldsymbol \varphi}

\def\dps{\displaystyle}

\renewcommand{\headrule}{{\color{gray}%
\hrule width\headwidth height\headrulewidth \vskip-\headrulewidth}}
\renewcommand{\headrulewidth}{.2pt}  

\makeatletter
\renewcommand\@maketitle{%
\hfill
\begin{minipage}{\textwidth}
\vskip -1em
\let\footnote\thanks
\begin{center}
  {\LARGE \@title \par }
  \vskip 1em
  { \@author \par}
    { \@date \par}
\end{center}
\headrule
\end{minipage}
\par
}
\makeatother

\usepackage{pifont} 

\newcommand{\hoathib}{{\color{RoyalBlue}\ding{118}}}
\newcommand{\muitenb}{{\color{RoyalBlue}\ding{42}}}

\newcommand{\hoab}{{\color{RoyalBlue}\ding{93}}}

\title{\color{RoyalBlue}\tbsi{A Result of Uniqueness of Solutions of the  Shigesada-Kawasaki-Teramoto Equations} }
 \author[1]{\tac{Du Pham}\thanks{\color{JungleGreen}du.pham@utsa.edu}}
 \author[2]{\tac{Roger Temam}\thanks{\color{JungleGreen}temam@indiana.edu}}
\affil[1]{\footnotesize\tac{Department of Mathematics, University of Texas at San Antonio, One UTSA Circle, San Antonio, Texas 78249, U.S.A.} }
\affil[2]{\footnotesize\tac{The Institute for Scientific Computing and Applied Mathematics, Indiana University, 831 East Third Street, Rawles Hall, Bloomington, Indiana 47405, U.S.A.}}

 \date{\empty}

\newcommand*{\boxcolor}{orange}
\makeatletter
\renewcommand{\boxed}[1]{\textcolor{\boxcolor}{%
\tikz[baseline={([yshift=-.6ex]current bounding box.center)}] \node [rectangle, minimum width=1ex,rounded corners,draw] {\normalcolor\m@th$\displaystyle#1$};}}
 \makeatother

\begin{document}

\usefont{T1}{cmdh}{m}{n}

\maketitle
\thispagestyle{fancy}

\cfoot[]{\tac{\large\color{gray}Page \thepage\ of \pageref{LastPage}}}
\rhead{\color{gray}\footnotesize{\tac{Uniqueness \& Wellposedness for SKT systems}}}
\lhead{\color{gray}{\footnotesize\tac{D. Pham \& R. Temam }}}

\begin{abstract} \lettrine{\color{RoyalBlue}W}{e} derive the uniqueness of weak solutions to the  Shigesada-Kawasaki-Teramoto (SKT) systems using the adjoint problem argument. Combining with \cite{PT16} we then derive the well-posedness for the SKT systems  in space dimension $d\le 4$.\end{abstract}

  {\footnotesize
  \paragraph{Keywords and phrases:}    wellposedness;   quasi-linear parabolic equations; global existence.

  \paragraph{2010 Mathematics Subject Classification:}   35K59, 35B40, 92D25. 
  }

\tableofcontents
%
\section{Introduction\label{sec:intro}}
\lettrine{\color{RoyalBlue}U}{niqueness} is an important issue to address when one considers the global well-posedness for a system  of  differential equations. For systems of partial differential equations like cross diffusion systems, the uniqueness has remained a challenge for solutions with mild regularity since the comparison/maximum principle for the cross diffusion systems like SKT is not available. We also note here that in our recent work \cite{PT16}, we showed a weak maximum principle for non-negativeness of solutions that allowed us to prove the existence of positive weak solutions of SKT systems directly using finite difference approximations, a priori estimates and passage to the limit, which avoid the change of variables (or entropy function) being used in other works as in e.g. \cite{CJ04,CDJ16,Jun15}. Together with our existence result for weak solutions of SKT systems in \cite{PT16}, this article provides the well-posedness for these systems in space dimension $d\le 4$.

The available uniqueness results for the KST systems are rather scarce and require high regularity of the solutions. In \cite[Theorem 3.5]{Yag93}, the author proved a uniqueness result for solutions in $\mc{C}((0,T];H^2(\Omega))\cap\mc{C}^1((0,T];L^2(\Omega))$ using an abstract theory for parabolic equations in space dimension 2. In \cite{Ama89,Ama90}, the author proved global existence and uniqueness results for solutions of general systems of parabolic equations with high regularity in space in the semigroup settings, $W^{1,p}(\Omega)$ for $p>n$, which require H\"{o}lder a priori estimates when applied to SKT equations.

In this work, we use the argument of adjoint problems to build specific test functions to show the uniqueness for solutions in a more general space setting, in $L^\infty(0,T;H^1(\Omega)^2)$ with time derivatives in $L^\f43(\Omega_T)^2$ for space dimension $d\le 4$, see Remark \ref{rmk:soln reg} below. This argument of using adjoint problems has been used to show uniqueness results for scalar partial differential differential equations describing flows of gas or fluid  in porous media or the spread of a certain biological population, see e.g. \cite{Aro85,ACP82}. It is also systematically used in the context of linear equations in \cite{LM72}.

Throughout our work, we denote by  $\Omega$   an open bounded domain in $\mathbb{R}^d$, with $d\le  4$, and we set $\Omega_T=\Omega\times (0,T)$ for any $T>0$.  We aim to show the uniqueness result and then combine this result to our previous global existence results of weak solutions in \cite{PT16} to show the global wellposedness for the following SKT system of diffusion reaction equations, see \cite{SKT79}:
\begin{equation} \label{SKT vector}
 \begin{cases}
  &\dps \partial_t \vu - \Delta \vp(\vu)  + \vq(\vu) = \vl(u)\text{ in } \Omega_T,\\
  &\partial_\nu \vu = \vo\text{ on } \partial \Omega \times (0,T) \text{ or }\vu=\vo \text{ on }\partial \Omega\times (0,T), \\
  &\vu(x,0) = \vu_0(x) \ge \vo, \text{ in } \Omega,
    \end{cases}
\end{equation}
where $\vu =(u,v)$ and
\begin{subequations}
 \label{pi qi li}
 \begin{align}
&\vp(\vu) = \begin{pmatrix}
  p_1(u,v) \\ p_2(u,v)
\end{pmatrix} = \begin{pmatrix}
  (d_1 + a_{11} u+ a_{12}v) u \\ (d_2 + a_{21} u+ a_{22}v) v
\end{pmatrix} , \\
&\vq(\vu) = \begin{pmatrix}
  q_1(u,v) \\ q_2(u,v)
\end{pmatrix}
=\begin{pmatrix}
(b_1u+c_1v)u
\\ (b_2u+c_2v)v
\end{pmatrix},
\\& \text{and }
\vl(\vu) =\begin{pmatrix}
  \ell_1(u)\\ \ell_2(v)
\end{pmatrix}
=\begin{pmatrix}
  a_1u\\  a_2v
\end{pmatrix}.
 \end{align}
\end{subequations}
Here $a_{ij} \ge 0, b_i \ge 0, c_i \ge 0, a_i \ge 0, d_i \ge 0$ are such that
\begin{equation}\label{1.5c}
  0<a_{12} a_{21} < 64 a_{11}a_{22}. 
\end{equation}
It can be shown in \cite{Yag08} that the condition \eqref{1.5c} is equivalent to
\begin{equation} \label{coef cond}
  0<a_{12}^2 < 8 a_{11}a_{21} \text{ and } 0<a_{21}^2 <8a_{22}a_{12}, 
\end{equation}
as far as existence and uniqueness of solutions are concerned.

One of the difficulties with the SKT equations is that they are not parabolic equations. Whereas Amann \cite{Ama89,Ama90}  has proven the existence and uniqueness of regular solutions for general parabolic equations, which can be applied to SKT equations using $L^p$ estimates, we proved in \cite{PT16} the existence of weak solutions (see also \cite{Jun15}) to the SKT equations. It is important, to validate this concept of weak solutions, to show that the weak solutions are unique. This is precisely what we are doing in this article in dimension $d\le 4$.

Throughout the article, we often use the following alternate form of \eqref{SKT vector}:
\begin{equation}
  \label{SKE alternate vector}
   \partial_t \vu - \grad \cdot \Big( \vP(\vu) \grad \vu \Big) + \vq (\vu) =\vl (\vu),
\end{equation}
where
\begin{equation}
 \label{P}
 \vP(\vu) = \begin{pmatrix}
             p_{11}(u,v) & p_{12}(u,v)\\
             p_{21}(u,v) & p_{22}(u,v)
            \end{pmatrix}
            =\begin{pmatrix} d_1 +2a_{11}u + a_{12}v & a_{12} u \\ a_{21}v & d_2 + a_{21}u +a_{22}v \end{pmatrix}.
\end{equation}

When the condition \eqref{coef cond} is satisfied and $u\ge 0, v\ge 0$, we can prove that the matrix $\vP(\vu)$ is (pointwise)  positive definite and that:
\begin{equation} \label{P positive definite}
\left(\vP(\vu)\vxi \right) \cdot \vxi \ge \alpha(u+v)\abs{\vxi}^2 + d_0 \abs{\vxi}^2,\quad \forall \vxi \in \mathbb{R}^2,
\end{equation}
where $d_0 = \min (d_1,d_2)$ and
 \begin{equation}\label{alpha}
   0<\alpha < \min \left(a_{11}, a_{12},a_{21 },a_{22},\delta_0\right) ;
   \end{equation}
Here we refer the readers to a proof of \eqref{P positive definite} in our recent article \cite{PT16}.

We consider later on the mappings
\begin{equation}
  \label{P Q map} \mc{P}: \vu=(u,v) \mapsto \vp=(p_1,p_2), \quad \mc{Q}: \vu=(u,v) \mapsto \vq=(q_1,q_2),
\end{equation}
and we observe that
\begin{equation}
  \label{P Q jac} \vP(\vu) = \f{D\mc{P}}{D \vu}(\vu),\quad  \vQ(\vu) = \f{D\mc{Q}}{D \vu}(\vu),
\end{equation}
and
\begin{equation}
  \label{P grad u} \grad \vp(\vu) = \vP(\vu) \grad \vu.
\end{equation}
We see that the explicit form of $\vP(\vu)$ is given  in \eqref{P} and that of $\vQ(\vu)$ is
 \begin{equation}
   \label{Q}\vQ(\vu) = \begin{pmatrix}
    2b_1 u+c_1 v & c_1 u \\ b_2v & b_2u+2c_2v
  \end{pmatrix}.
 \end{equation}

Note that \eqref{P positive definite} implies that, for $u,v \ge 0$, $\vP(\vu)$ is invertible (as a $2\times 2$ matrix),
and that, pointwise (i.e. for a.e. $x\in \Omega$),
\begin{equation}\label{2.5b}
 \abs{\vP(\vu)^{-1}}_{\mathcal{L}(\mathbb{R}^2)} \le \f{1}{d_0+\alpha(u+v)}.
\end{equation}
Our work is organized as follows. We show our main result in Section \ref{sec:uniq}, where the uniqueness for weak solutions to the SKT system is derived using solutions of adjoint problems. Since the proof of the uniqueness relies on the existence of solutions to the adjoint problem, we show the existence for these problems in Section \ref{sec:adj}, together with the apriori estimates in dimension $d\le4$. We finally show in Section \ref{sec:global} that the newly derived uniqueness result in Section \ref{sec:uniq} together with our existence result in \cite{PT16} leads to the global well-posedness for the SKT systems in space dimension $d\le 4$.
\section{Uniqueness result for SKT systems\label{sec:uniq}}

\lettrine{\color{RoyalBlue}A}{s} mentioned earlier, our uniqueness result  is proven using an argument of an adjoint problem, see e.g. \cite{ACP82}; see also \cite{LM72} in the context of linear parabolic problems. The existence of solution of our adjoint problem will be granted if the solution $\vu$ of \eqref{SKT vector}  enjoys the following regularity properties
\begin{equation}
  \label{uniqueness cond}
  \vu \in L^\infty(0,T;H^1(\Omega)^2),  \text{ and }\partial_t \vu \in L^\f43(\Omega_T)^2.
\end{equation}
\begin{Rmk}Although this was not explicitly stated in \cite{PT16}, the solutions that we constructed in dimension $d\le 4$ belong to $L^\infty(0,T;H^1(\Omega)^2)$ with $\partial_t \vu \in L^2(0,T;L^2(\Omega)^2)$; see Appendix \ref{appen: C}.\label{rmk:soln reg}

\end{Rmk}

Introducing a test function $\vphi$ which satisfies \eqref{phi} below and the same boundary  condition as $\vu$, we multiply \eqref{SKT vector} by $\vphi$, integrate, integrate by parts and obtain the variational weak form of \eqref{SKT vector}:
\begin{equation}
  \label{ui eq}
  \begin{cases}
   &\dps \inner{\partial_t \vu}{\vphi} - \inner { \vp(\vu)}{\Delta \vphi}  + \inner{\vq(\vu)}{\vphi} = \inner{\vl(\vu)}{\vphi},\\
   &\partial_\nu \vu = \vo\text{ on } \partial \Omega \times (0,T) \text{ or }\vu=\vo \text{ on }\partial \Omega\times (0,T), \\
   &\vu(x,0) =  \vu_0, \text{ in } \Omega,
     \end{cases}
\end{equation}
for all  test functions $\vphi $ such that
\begin{equation}
 \label{phi}
 \begin{cases}
   &\vphi \in L^2(0,T;H^2(\Omega)^2) \cap L^\infty(0,T;H^1(\Omega)^2), \text{ and }\partial_t\vphi \in L^\f43(\Omega_T)^2, \\
   &\partial_\nu \vphi = \vo \text{ or } \vphi =\vo \text{ on }\partial \Omega \text{ ($\vphi$ satifies the same b.c. as $\vu$)}.
 \end{cases}
\end{equation}
Note that the boundary terms disappear because $\vp(\vu)$ satisfies the same b.c. as $\vu$.
To show that the solutions of \eqref{SKT vector} are unique, we introduce the difference of two solutions $\vu_1,\,\vu_2$ of \eqref{ui eq},  $\bar{\vu} = \vu_1 -\vu_2$, and we will eventually show that  $\bar{\vu}=\vo$ for a.e. $\vx\in \Omega$ and $t>0$.

 We first observe  that $\bar{\vu}$ satisfies
\begin{equation}
  \label{u diff eq}
  \begin{cases}
   &\dps \inner {\partial_t \bar{\vu}}{\vphi} - \inner{\vp(\vu_1)-\vp(\vu_2)}{\Delta \vphi}  + \inner{\vq(\vu_1)-\vq(\vu_2)}{\vphi} = \inner{\vl(\vu_1) -\vl(\vu_2)}{\vphi},\\
   &\partial_\nu \bar{\vu} = \vo\text{ on } \partial \Omega \times (0,T) \text{ or }\bar{\vu}=\vo \text{ on }\partial \Omega\times (0,T), \\
   &\bar{\vu}(x,0) =  \vo, \text{ in } \Omega,
     \end{cases}
\end{equation}
for any test function  $\vphi$ that satisfies \eqref{phi}.

Using the notations $\vP(\star),\,\vQ(\star)$ introduced earlier in \eqref{P Q jac} and the relations  \eqref{p diff}, \eqref{q diff} from Lemma \ref{lem: u diff}, we find
  \[\inner{\bar{\vu}}{\vphi}_t-\inner{\bar{\vu}}{\vphi_t}-\inner{\vP(\tilde{\vu})\bar{\vu}}{\Delta\vphi} +\inner{\vQ(\tilde{\vu})\bar{\vu}}{\vphi} = \inner{\vl(\bar{\vu})}{\vphi},\]
where $\tilde{\vu} = (\vu_1+\vu_2)/2$.

 Thus
\begin{equation}\label{skt inner}
  \inner{\bar{\vu}}{\vphi}_t-\inner{\bar{\vu}}{\vphi_t}-\inner{\bar{\vu}}{\vP(\tilde{\vu})^T\Delta\vphi} +\inner{\bar{\vu}}{\vQ(\tilde{\vu})^T\vphi} = \inner{\vl(\bar{\vu})}{\vphi}.
\end{equation}
We notice that $\tilde\vu\in L^\infty(0,T;H^1(\Omega)^2)$ because $\vu_1,\vu_2\in L^\infty(0,T;H^1(\Omega)^2)$.

We now consider the test function $ \vphi $ to be solution of the following backward adjoint problem
\begin{equation}
  \label{adj eq phi}
  \begin{cases}
    &-\partial_t\vphi - \vP(\tilde{\vu})^T  \Delta\vphi+ \vQ(\tilde{\vu})^T\vphi = \vphi \text{ in }\Omega_T,
    \\& \partial_\nu \vphi= \vo \text{ or } \vphi = \vo  \text{ on }\partial \Omega\times(0,T),
    \\& \vphi(T) =\vchi(\vx)\text{ in }\Omega,
  \end{cases}
\end{equation}
where $\vchi(\vx)= (\chi^u(\vx),\chi^v(\vx)) \in H^1(\Omega)^2$.

Before showing the uniqueness result of solutions of  \eqref{SKT vector} using  the test function $\vphi$ as a solution of  \eqref{adj eq phi}, we first show the existence of  $\vphi =(\phi^u,\phi^v)\in L^2(0,T;H^2(\Omega)^2)$  with $\partial_t \vphi \in L^\f43(\Omega_T)$ in the following section.

\subsection{Existence of solutions for the adjoint systems \label{sec:adj}}
\lettrine{\color{RoyalBlue}I}{n} this section, we continue to assume that $d\le4$ and we show the existence of a solution $\vphi$ of \eqref{adj eq phi} satisfying \eqref{phi} by building approximate systems where the classical existence theory can be applied to show the existence of approximate solutions. We suppose throughout this section that the functions $\tilde\vu \ge \vo$ in \eqref{adj eq phi}$_1$ satisfies
\begin{equation}\label{tilde u}
  \tilde\vu \in L^\infty(0,T;H^1(\Omega)^2).
\end{equation}

We observe that the diffusive matrix $\vP(\tilde\vu)$ in \eqref{adj eq phi} may not be uniformly parabolic\footnote{$\vP(\star)$ is uniformly parabolic if $ \kappa_1 \abs{\vxi}^2 \le \left(\vP(\star)\vxi \right) \cdot \vxi  \le \kappa_2 \abs{\vxi}^2,$ for some $\kappa_1,\kappa_2>0$; see the definition in e.g. \cite{Lad68}} unless $\tilde\vu \in L^\infty(\Omega_T)^2$. Thus, we can not directly apply the classical results for parabolic equations to show the existence of $\vphi$, see e.g. \cite[Theorem 5.1]{Lad68}. We therefore use an approximation approach as in \cite{ACP82}.


The existence of solution $\vphi$ of \eqref{adj eq phi} is obtained in three steps:
\begin{itemize}
  \item[\muitenb] Define approximations $\vphi_\ep$ of $\vphi$, which are solutions of the approximate systems \eqref{adj eq ep} below.
  \item[\muitenb] Derive a priori estimates for the functions $\vphi_\ep$.
  \item[\muitenb] Pass to the limit as $\ep \rightarrow 0$ to show the existence of $\vphi$, solution  of \eqref{adj eq phi}.
\end{itemize}
We start now with the first step of building approximate solutions $\vphi_\ep$:


\subsubsection{Approximate adjoint systems}  We know that  $\tilde\vu= (\vu_1+\vu_2)/2 \ge \vo $ and $\tilde\vu \in L^\infty(0,T;H^1(\Omega)^2)$ (see Theorem \ref{thm: existence}).  We build approximations $\tilde\vu_\ep$ of $\tilde\vu$, as a sequence in $L^\infty(\Omega_T)^2$, that converges to $\tilde\vu$ in $L^4(\Omega_T)^2$. We can define such $\tilde\vu_\ep$ as  follows
\begin{equation}
  \label{u ep}
  \tilde\vu_\ep = \vtheta_\ep(\tilde\vu),
\end{equation}
where $\vtheta_\ep $ is a smooth function with derivative bounded by a constant independent of $\ep$, which we assume to be $1$, such that
\begin{equation}
 \label{theta ep} \vtheta_\ep(\tilde{\vu} ) =\begin{cases}
   & \tilde{\vu} \text{ for }  \tilde{\vu} \le \f{1}{\ep},
 \\  & \f{1}{\ep} \text{ for } \tilde{\vu} \ge \f{2}{\ep}.
\end{cases}
\end{equation}
We easily see that $\vu_\ep \in L^\infty(\Omega_T)^2$. Furthermore, we have $\tilde\vu_\ep \rightarrow \tilde \vu$ a.e. and $\abs{\tilde\vu_\ep}_{L^4} \le \abs{\tilde\vu}_{L^4}<\infty$, and we obtain by the Lebesque dominated convergence that $\tilde\vu_\ep$ converges to $\tilde\vu$ in $L^4(\Omega_T)^2$. Finally, we easily see the following by straightforward calculations
\begin{equation}\label{u ep bounded by u}
 \norm{\tilde\vu_\ep}_{L^\infty(0,T;H^1(\Omega)^2)} \le \kappa   \norm{\tilde\vu}_{L^\infty(0,T;H^1(\Omega)^2)},
\end{equation}
where $\kappa$ depends on the maximum value of $\theta_\ep'$ which is  independent of $\ep$.

 We then let $\vphi_\ep = (\varphi^\ep_u, \varphi^\ep_v) $ satisfy the following approximate system
\begin{equation}
  \label{adj eq ep}
  \begin{cases}
    &-\partial_t\vphi_\ep - \vP(\tilde{\vu}_\ep)^T  \Delta\vphi_\ep+ \vQ(\tilde{\vu}_\ep)^T\vphi_\ep = \vphi_\ep \text{ in }\Omega_T,
    \\& \partial_\nu \vphi_\ep = \vo \text{ or } \vphi_\ep = \vo  \text{ on }\partial \Omega\times(0,T),
    \\& \vphi_\ep(T) =\vchi(\vx)\text{ in }\Omega.
  \end{cases}
\end{equation}
 We know that $\tilde\vu_\ep \in L^\infty(\Omega_T)^2$ which yields $\vP(\tilde\vu_\ep) \in L^\infty(\Omega_T)^4$, and this in turn implies that
\[\left(\vP(\tilde\vu_\ep)\vxi \right) \cdot \vxi  \le \kappa(\ep) \abs{\vxi}^2,\]
where $\kappa(\ep)$ is a constant depending on $\ep$. This bound from  above of $\vP(\tilde{\vu}_\ep)$ and its bound from the below in \eqref{P positive definite} give the uniform parabolic condition for the approximate system \eqref{adj eq phi}. The existence of a smooth function $\vphi_\ep$ is hence given by the classical theory of the equations of parabolic type, see e.g. \cite[Theorem 5.1]{Lad68}.
We now bound  the approximate solutions $\vphi_\ep$ independently of $\ep$:

\begin{Lem}\label{lem: phi bound} Assume that $d\le4$ and $\tilde\vu\in L^\infty (0,T;H^1(\Omega)^2)$.
  We then have the following a priori bounds independent of $\ep$ for the solution $\vphi_\ep$ of \eqref{adj eq ep}:
\begin{subequations}
  \begin{align}
      & \label{phi bound} \sup_{t\in [0,T]} \norm{\vphi_\ep}_{H^1(\Omega)^2} \le \kappa \norm{\vchi}_{H^1(\Omega)^2}, \\
      & \int_0^T (1+\tilde u_\ep+\tilde v_\ep) \abs{\Delta \vphi_\ep }^2 dt \le \kappa \norm{\vchi}_{H^1(\Omega)^2}, \label{phi bound1}
    \end{align}
and 
\begin{equation}\label{time der bound}
      \norm{\partial_t \vphi_\ep}_{  L^\f43(\Omega_T)} \le \kappa \norm{\vchi}.
\end{equation}
\end{subequations}
Here, in this lemma, $\kappa$  depends on $\norm{\vu}_{ L^\infty (0,T;H^1(\Omega)^2)}$ and on the coefficients but is independent of $\ep$.
\end{Lem}
\paragraph{Proof of Lemma \ref{lem: phi bound}:}
    Multiplying \eqref{adj eq ep} by $ \vphi_\ep$, we find
      \begin{equation}\label{phi 1}
        -\f12 \f{d}{dt}\abs{ \vphi_\ep}^2 - \inner{\vP(\tilde{\vu}_\ep)^T\Delta \vphi_\ep}{ \vphi_\ep} +  \inner{\vQ(\tilde{\vu}_\ep)^T\vphi_\ep}{\vphi_\ep} =\abs{ \vphi_\ep}^2.
      \end{equation}
  Multiplying \eqref{adj eq ep} by $-\Delta \vphi_\ep$, we also find after integration by parts
\begin{equation}\label{phi 2}
  -\f12 \f{d}{dt}\abs{\grad \vphi_\ep}^2 + \inner{\vP(\tilde{\vu}_\ep)^T\Delta \vphi_\ep}{\Delta \vphi_\ep} -  \inner{\vQ(\tilde{\vu}_\ep)^T\vphi_\ep}{\Delta \vphi_\ep} =\abs{\grad \vphi_\ep}^2
\end{equation}
Adding equations \eqref{phi 1} and \eqref{phi 2} and regrouping the terms, we find
\begin{multline} \label{phi 3}
  -\f12\f{d}{dt}\left( \abs{\vphi_\ep}^2+\abs{\grad \vphi_\ep}^2\right) + \inner{\vP(\tilde{\vu}_\ep)^T\Delta \vphi_\ep}{\Delta \vphi_\ep }  \\= \inner{\Big(\vP(\tilde{\vu}_\ep) + \vQ(\tilde\vu_\ep)^T\Big)\vphi_\ep}{\Delta \vphi_\ep } - \inner{\vQ(\tilde{\vu}_\ep)^T\vphi_\ep}{ \vphi_\ep}+ \abs{\vphi_\ep}^2+\abs{\grad \vphi_\ep}^2.
\end{multline}
We bound the first two terms on the right hand side of \eqref{phi 3} as follows:
\begin{itemize}
  \item[\hoathib] We first bound the easier term $\inner{\vQ(\tilde{\vu}_\ep)^T\vphi_\ep}{ \vphi_\ep}$ using the H\"{o}lder inequality for three functions with powers $(2,4,4)$, the Sobolev embedding from $H^1$ to $L^4$ in dimension $d\le 4$, and \eqref{u ep bounded by u}:
  \begin{align*}
 &    \abs{\inner{\vQ(\tilde{\vu}_\ep)^T\vphi_\ep}{ \vphi_\ep}} \le c_0\abs{\tilde\vu_\ep}_{L^2} \abs{\vphi_\ep}_{L^4}^2 \le c_1 \norm{\tilde\vu_\ep}_{H^1} \norm{ \vphi_\ep}_{H^1}^2 \le \kappa_1\left( \norm{\tilde\vu}_{L^\infty(0,T;H^1)}\right) \norm{ \vphi_\ep}_{H^1}^2  .
  \end{align*}
Here $ \kappa_1\left( \norm{\tilde\vu}_{L^\infty(0,T;H^1)}\right) $ is a constant which depends on $ \norm{\tilde\vu}_{L^\infty(0,T;H^1)} $ but not on $\ep$.
  \item[\hoathib] We now bound the term $\inner{\Big(\vP(\tilde{\vu}_\ep) + \vQ(\tilde\vu_\ep)^T\Big)\vphi_\ep}{\Delta \vphi_\ep } $ using H\"{o}lder's inequality for three functions with powers $(4,4,2)$, the previously used  Sobolev embedding from $H^1$ to $L^4$ which assumes $d\le 4$, the Young inequality, and  \eqref{u ep bounded by u}:
  \begin{multline*}
     \abs{\inner{\Big(\vP(\tilde{\vu}_\ep) + \vQ(\tilde\vu_\ep)^T\Big)\vphi_\ep}{\Delta \vphi_\ep }} \le c_0 \abs{\tilde\vu_\ep}_{L^4} \abs{\vphi_\ep}_{L^4} \abs{\Delta\vphi_\ep}_{L^2}
    \\  \le c_1 \norm{\tilde\vu_\ep}_{H^1} \norm{\vphi_\ep}_{H^1} \abs{\Delta\vphi_\ep}_{L^2} \hspace{4cm}
    \\  \le c_1 \norm{\tilde\vu_\ep}_{L^\infty(0,T;H^1)} \norm{\vphi_\ep}_{H^1} \abs{\Delta\vphi_\ep}_{L^2}
    \\  \le \kappa_2 \left(\norm{\tilde\vu}_{L^\infty(0,T;H^1)}\right)\norm{\vphi_\ep}_{H^1}^2 +  \f{d_0}{2}\abs{\Delta\vphi_\ep}_{L^2} ^2,
  \end{multline*}
\end{itemize}
where $d_0=\min(d_1,d_2)$  and $ \kappa_2\left( \norm{\tilde\vu}_{L^\infty(0,T;H^1)}\right) $ is independent of $\ep$.

Using these two bounds in \eqref{phi 3}, we find
\begin{multline*}
     -\f12\f{d}{dt}\left( \abs{\vphi_\ep}^2+\abs{\grad \vphi_\ep}^2\right) + \inner{\vP(\tilde{\vu}_\ep)^T\Delta \vphi_\ep}{\Delta \vphi_\ep }  \\\le  \kappa \left(\norm{\tilde\vu}_{L^\infty(0,T;H^1)}\right) \left( \abs{\vphi_\ep}^2+\abs{\grad \vphi_\ep}^2\right) +  \f{d_0}{2}\abs{\Delta\vphi_\ep}_{L^2} ^2,
\end{multline*}
where  $ \kappa\left( \norm{\tilde\vu}_{L^\infty(0,T;H^1)}\right) $ is a constant which depends on $ \norm{\tilde\vu}_{L^\infty(0,T;H^1)} $ but is independent of $\ep$.

Thanks to the positivity of $\vP(\star)$ in \eqref{P}, we  have using \eqref{P positive definite}
\begin{multline}
 \label{phi 4}
      -\f12\f{d}{dt}\left( \abs{\vphi_\ep}^2+\abs{\grad \vphi_\ep}^2\right) + \left(\f{d_0}{2} +\alpha(\tilde{u}_\ep+\tilde v_\ep)\right) \abs{\Delta \vphi_\ep}^2
      \\ \le  \kappa \left(\norm{\tilde\vu}_{L^\infty(0,T;H^1)}\right) \left( \abs{\vphi_\ep}^2+\abs{\grad \vphi_\ep}^2\right)  .
\end{multline}
This implies
\begin{equation}
  \label{3.10b}
    -\f{d}{dt}\left( \abs{\vphi_\ep}^2+\abs{\grad \vphi_\ep}^2\right) + \alpha(1+\tilde u_\ep +\tilde v_\ep)\abs{\Delta \vphi_\ep}^2  \le \kappa \left( \abs{\vphi_\ep}^2+\abs{\grad \vphi_\ep}^2 \right),
\end{equation}
where again $\kappa = \kappa \left(\norm{\tilde\vu}_{L^\infty(0,T;H^1)}\right)$.

Recall that $\vphi(T) = \vchi \in H^1(\Omega)^2$. By multiplying \eqref{3.10b} by $e^{2t}$ and integrating over $[t,T]$ for $t\in [0,T]$, we infer \eqref{phi bound} and \eqref{phi bound1}.

Now, to derive the bound  independent of $\ep$ for $\partial_t\vphi_\ep$, we  write using \eqref{adj eq ep}$_1$:
\begin{equation}\label{eq}
 \norm{\partial_t\vphi_\ep }_{L^\f43} = \norm{\vP(\tilde{\vu}_\ep)^T\Delta \vphi_\ep - \vQ(\tilde{\vu}_\ep)^T\vphi_\ep +\vphi_\ep }_{L^\f43}.
\end{equation}
We bound the most challenging norm term $\norm{\vP(\tilde{\vu}_\ep)^T\Delta \vphi_\ep}_{L^\f43} $ on the right hand side of \eqref{eq}. We consider a function $\vz \in L^4(\Omega_T)^2$ and write
\begin{multline*}
\int_{\Omega_T}\vP(\tilde{\vu}_\ep)^T\Delta \vphi_\ep \vz\,dxdt \\
 \le \left( \max_{i=1,2}d_i \abs{\Omega}^\f14 + 2\max_{i,j=1,2}a_{ij} \Big(\norm{u_\ep}_{L^4}+\norm{v_\ep}_{L^4}\Big)\right) \left(\norm{\Delta \varphi^u_\ep}_{L^2} + \norm{\Delta \varphi^v_\ep}_{L^2}\right) \norm{\vz}_{L^4}.
\end{multline*}
Observing that $\norm{u_\ep}_{L^4} \le  \norm{u}_{L^4} \le \norm{u}_{L^\infty(0,T;H^1)}$, a similar bound for $\norm{v_\ep}_{L^4}$,  and using \eqref{phi bound1}, we find the following bound for any $\vz \in L^4(\Omega_T)$
\[\int_{\Omega_T}\vP(\tilde{\vu}_\ep)^T\Delta \vphi_\ep \vz\,dxdt \le \kappa(\norm{\vu}_{L^4}) \abs{\Delta \vphi_\ep}\norm{\vz}_{L^4}\le \kappa(\norm{\vu}_{L^\infty(0,T;H^1)}) \norm{\vchi}\norm{\vz}_{L^4},\]
where $\kappa$  depends on $\norm{\vu}_{L^\infty(0,T;H^1)}$ and on the coefficients $d_i, a_{ij}$ but is independent of $\ep.$ This gives the a priori bound \eqref{time der bound}.
\hfill \qed
\subsubsection{Passage to the limit for the solutions of the approximate systems}

We now pass to the limit as $\ep \rightarrow 0$ in the approximate adjoint system \eqref{adj eq ep}. We have from the a priori estimates \eqref{phi bound} -- \eqref{time der bound} that there exist a subsequence of $\vphi_\ep$, still denoted by $\vphi_\ep$, such that as $\ep \rightarrow 0$
\begin{subequations}
 \label{phi conv}
  \begin{align}
   &\text{\hoab}\quad\vphi_\ep \rightharpoonup\vphi \text{ in }L^\infty(0,T;H^1(\Omega)^2) \text{ weak-star},\\
 \label{phi convb}&\text{\hoab}\quad\Delta\vphi_\ep \rightharpoonup  \Delta \vphi\text{ in }L^2(\Omega_T)^2 \text{ weakly},\\
   \label{phi conv}  &\text{\hoab}\quad\partial_t \vphi_\ep \rightharpoonup\partial_t\vphi\text{ in }L^\f43(\Omega_T)^2 \text{ weakly}.
  \end{align}
\end{subequations}
We then pass to the limit term by term in  \eqref{adj eq ep}, where the most challenging product term $\vP(\tilde\vu_\ep)\Delta\vphi_\ep$ is treated as follows: for any $\vz\in L^4(\Omega_T)^2$, we write
\begin{multline}
 \int_{\Omega_T}\left(\vP(\tilde\vphi_\ep)^T\Delta\vphi_\ep - \vP(\tilde\vphi)^T\Delta\vphi\right)\vz \,dxdt \\
 = \ubrace{\int_{\Omega_T}\Delta\vphi_\ep\left(\vP(\tilde\vu_\ep) - \vP(\tilde\vu)\right)\vz \,dxdt}{T_\ep^1} + \ubrace{\int_{\Omega_T}\left(\Delta\vphi_\ep - \Delta\vphi\right)\vP(\tilde\vu)\vz \,dxdt }{T_\ep^2}.
\end{multline}
\begin{enumerate}[\hoab\;]
 \item We first deal with the easier term $T_\ep^2$. We easily see that $T_\ep^2 \rightarrow 0$ as $\ep\rightarrow0$ thanks to \eqref{phi convb} and the facts that $\vP(\tilde\vu)\in L^4(\Omega_T)^4,\vz\in L^4(\Omega_T)^2$ which gives  $\vP(\tilde\vu)\vz\in L^2(\Omega_T)^2$.
 \item For $T_\ep^1$, as $\ep \rightarrow 0$, we know that $\tilde\vu_\ep \rightarrow \tilde\vu$ in $L^4(\Omega_T)^2$ strongly which gives $\vP(\tilde\vu_\ep) \rightarrow  \vP(\tilde\vu)$ in $L^4(\Omega_T)^4$ strongly. Thus $\left(\vP(\tilde\vu_\ep) - \vP(\tilde\vu)\right)\vz$ converges to $\vo$ in $L^2(\Omega_T)^2$ strongly. Since $\Delta\vphi_\ep$ is bounded in $L^2(\Omega_T)^2$ (thanks to \eqref{phi bound}$_2$), we conclude that $T_\ep^1\rightarrow 0$ as $\ep\rightarrow 0$.
\end{enumerate}
We hence have $\vP(\tilde\vu_\ep)\Delta\vphi_\ep \rightharpoonup\vP(\tilde\vu)\Delta\vphi$ weakly in $L^\f43(\Omega_T)^2$ as $\ep\rightarrow 0$. 
We thus conclude that $\vphi$ is a weak solution of \eqref{adj eq phi} as below
\begin{Prop} \label{thm: phi exist} Under the assumptions that $d\le 4$ and $\tilde\vu \in L^\infty(0,T;H^1(\Omega)^2)$, the adjoint system \eqref{adj eq phi} admits a  solution $\vphi$ in $L^\infty(0,T;H^1(\Omega)) \cap L^2(0,T;H^2(\Omega)^2)$  such that $\partial_t \vphi \in L^\f43(\Omega_T)^2$.
\end{Prop}

We now resume the work of showing the uniqueness of solution $\vu$ of \eqref{SKT vector}:
\subsection{Uniqueness result for the SKT system} We recall  that $\vu_1,\,\vu_2$ are two solutions  of \eqref{ui eq}, and we have written $\bar{\vu} = \vu_1 -\vu_2$; we will eventually show that  $\bar{\vu}=\vo$ for a.e. $\vx\in \Omega$ and $t>0$.
 We also recall that that $\bar{\vu}$ satisfies
  \eqref{u diff eq}.

Using the existence result of the adjoint problem in Theorem \ref{thm: phi exist}, we have  a solution $\vphi =(\phi^u,\phi^v)\in L^2(0,T;H^2(\Omega)^2)$ of \eqref{adj eq phi} with $\partial_t \vphi \in L^\f43(\Omega_T)$.

We henceforth infer from \eqref{skt inner} and \eqref{adj eq phi} that
\begin{equation}\label{eq u bar phi}
  \inner{\bar{\vu}}{\vphi}_t+\inner{\bar{\vu}}{\vphi} = \inner{\vl(\bar{\vu})}{\vphi}.
\end{equation}
We look at the first component in
 \eqref{eq u bar phi}:
 \[\inner{u_1(t)-u_2(t)}{\varphi^u}_t = (a_1-1) \inner{u_1(t)-u_2(t))}{\varphi^u} .\]
Multiplying the equation by $e^{-(a_1-1)t}$ and integrating over the time interval $[0,T]$, we find
 \[\inner{u_1(T)-u_2(T)}{\chi^u} = 0.\]
 This is true for any $\chi^u\in H^1(\Omega)$ and we thus find $u_1(T)=u_2(T)$ for a.e $\vx\in \Omega$. The argument is also valid for any other time $t<T$ which gives $u_1(t)=u_2(t)$ a.e..

 Similarly, we have $v_1(t) = v_2(t)$ a.e..

We  have thus shown the following result:
 \begin{Thm}[Uniqueness] In space dimension $d\le 4$, the SKT system  \eqref{SKT vector} admits at most one weak solution $\vu\ge \vo$ such that
\label{thm:uniq}     \[\vu\in L^\infty(0,T;H^1(\Omega)^2), \text{ and }\partial_t \vu \in L^\f43(\Omega_T)^2.\]

 \end{Thm}

\section{Global well-posedness for the SKT system\label{sec:global}} \lettrine{\color{RoyalBlue}I}{n} this section, we assume that $d\le 4$ and show that our uniqueness result in Section \ref{sec:uniq} yields the global well-posedness for solutions of the SKT system \eqref{SKT vector} with the following initial datum conditions
\begin{equation}
  \label{u0 cond}
  \vu_0 \in L^2(\Omega)^2 \text{ and } \grad \vp(\vu_0) \in L^2(\Omega)^4.
\end{equation}
\begin{Rmk}
  Thanks to the existence result in our prior work \cite{PT16}, whose main result is stated as Theorem \ref{thm: existence},  we see that for all $T>0$, under the assumptions that the space dimension $d\le 4$ and the initial data satisfies \eqref{u0 cond}, the SKT system  \eqref{SKT vector} possesses solutions $\vu \in L^\infty(0,T;H^1(\Omega)^2)$ with $\partial_t\vu \in L^\f43(\Omega_T)^\f43$ as consequences of \eqref{3.38c} and \eqref{3.36}. Theorem \ref{thm:uniq} thus applies and gives the uniqueness of such a solution $\vu$ of \eqref{SKT vector}.   We then conclude that the solution $\vu$ exists globally and uniquely.
\end{Rmk}
Our main result in this section is as follows:
\begin{Thm}Suppose that $d\le 4$, that $\vu_0$ satisfies \eqref{u0 cond}, and that the coefficients satisfy \eqref{coef cond}. The system \eqref{SKT vector} possesses a unique global solution $\vu \in L^\infty(0,\infty;H^1(\Omega)^2) $ with $\partial_t \vu \in L^2(\Omega\times(0,\infty))^2$.
Furthermore, the mapping $\vu_0 \mapsto \vu$ is continuous from $L^q(\Omega)$ into $L^2(\Omega)$ endowed with the norm $\abs{\star}_w$
\[\abs{\star}_w = \sup_{v\in H^1} \f{\inner{\star}{v}}{\norm{v}}.\]
Here $q=\max(2d/(6-d),4d/(d+2))$.
\end{Thm}

To show the continuous dependance on the initial data, we suppose that $\vu_1$ and $\vu_2$ are two solutions with initial data $\vu_1(0),\vu_2(0) $ satisfying \eqref{u0 cond}. We proceed as in Section \ref{sec:uniq} by denoting $\bar\vu = \vu_1 -\vu_2$, $\tilde \vu = (\vu_1+\vu_2)/2$, and recall from \eqref{skt inner} that
\begin{equation}\label{skt inner2}
  \inner{\bar{\vu}}{\vphi}_t-\inner{\bar{\vu}}{\vphi_t}-\inner{\bar{\vu}}{\vP(\tilde{\vu})^T\Delta\vphi} +\inner{\bar{\vu}}{\vQ(\tilde{\vu})^T\vphi} = \inner{\vl(\bar{\vu})}{\vphi},
\end{equation}
where $\vphi $ solves the following adjoint problem
\begin{equation}
  \label{adj eq phi2}
  \begin{cases}
    &-\partial_t\vphi - \vP(\tilde{\vu})^T  \Delta\vphi+ \vQ(\tilde{\vu})^T\vphi = \vl(\vphi) \text{ in }\Omega_\tau =\Omega \times (0,\tau),
    \\& \partial_\nu \vphi= \vo \text{ or } \vphi = \vo  \text{ on }\partial \Omega\times(0,\tau),
    \\& \vphi(\tau) =\vchi\text{ in }\Omega,
  \end{cases}
\end{equation}
for $\vchi(\vx)= (\chi^u(\vx),\chi^v(\vx)) \in H^1(\Omega)^2$ (arbitrary) with appropriate compatible boundary conditions.

The existence of solution that satisfies the following a priori estimates was proven in Lemma  \ref{lem: phi bound}  in Section \ref{sec:adj}:

\begin{Lem}[A priori estimates]\label{lem: phi bound2} Assume that $d\le4$ and $\tilde\vu  =(\tilde u,\tilde v)\in L^\infty (0,T;H^1(\Omega)^2)$.
  We then have the following a priori bounds independent of $\tau\in [0,T]$ for the solutions $\vphi$ of \eqref{adj eq phi2}:
\begin{subequations}
  \begin{align}
      & \label{phi bounda} \sup_{t\in [0,\tau]} \norm{\vphi}_{H^1(\Omega)^2} \le \kappa \norm{\vchi}_{H^1(\Omega)^2}, \\
      & \int_0^\tau (1+\tilde u+\tilde v) \abs{\Delta \vphi }^2 dt \le \kappa \norm{\vchi}_{H^1(\Omega)^2}, \label{phi boundb}
    \end{align}
\end{subequations}
Here, $\kappa$  depends on $\norm{\tilde\vu}_{ L^\infty (0,T;H^1(\Omega)^2)},T$ and on the coefficients but is independent of $\tau$.
\end{Lem}
We now continue to show the continuous dependance of $\vu$ on the data. We find from equations \eqref{skt inner2} and \eqref{adj eq phi2} that
\begin{equation}\label{eq u bar phi2}
  \inner{\bar{\vu}}{\vphi}_t = 0.
\end{equation}
Thus
\[\inner{\bar\vu(\tau)}{\vchi} = \inner{\bar\vu(0)}{\vphi(0)},\]
where we have used $\vphi(\tau) =\vchi$.

We now use \eqref{adj eq phi2}$_1$ and find
\begin{align}\nonumber
  \inner{\bar\vu(\tau)}{\vchi} & = \inner{\bar\vu(0)}{ \vchi + \int_0^\tau \big[ \vP(\tilde\vu)^T \Delta \vphi -\vQ(\tilde\vu)\vphi + \vl(\vphi)\big]dt}
    \\& \label{4.5}= \inner{ \bar\vu(0)} {\vchi} + \inner {\bar\vu(0)} { \int_0^\tau \big[\vP(\tilde\vu)^T \Delta \vphi -\vQ(\tilde\vu)\vphi + \vl(\vphi)\big] dt}
\end{align}
We next bound the typical terms on the RHS of \eqref{4.5}:
\begin{enumerate}[\hoab\;]
  \item We first bound a typical term $ \inner{\bar u(0)}{\tilde u \Delta \phi^u}$ in $\inner{\bar\vu(0)}{\vP^T(\tilde \vu) \Delta \vphi}$ as follows:
\begin{multline*}
   \inner{\bar u(0)}{\tilde u \Delta \phi^u} =    \inner{\ubrace{\bar u(0)}{\in L^\f{4d}{d+2}}}{\ubrace{\tilde u^\f12}{\in L^\f{4d}{d-2}}  \ubrace{\tilde u^\f12  \Delta \phi^u} {\in L^2} } \le \abs{\bar u(0)}_{L^\f{4d}{d+2}} \abs{\tilde u}_{L^\f{2d}{d-2}}^\f12 \abs{\tilde u^\f12 \Delta \phi^u}_{L^2}
   \\ \le c \abs{\bar u(0)}_{L^\f{4d}{d+2}} \abs{\tilde u}_{L^\infty(0,T;H^1)}^\f12 \abs{\tilde u^\f12 \Delta \phi^u}_{L^2}.
\end{multline*}
Thus, by \eqref{phi boundb}, we find
\begin{multline}\label{4.7}
   \inner{\bar \vu(0)}{\int_0^\tau\tilde \vu \Delta \vphi \,dt} \le c \abs{\bar \vu(0)}_{L^\f{4d}{d+2}} \abs{\tilde \vu}_{L^\infty(0,T;H^1)}^\f12 \int_0^\tau \abs{\tilde \vu^\f12 \Delta \vphi}_{L^2}
   \\ \le c T^\f12 \abs{\bar \vu(0)}_{L^\f{4d}{d+2}} \abs{\tilde \vu}_{L^\infty(0,T;H^1)}^\f12 \norm{\vchi}_{H^1}^\f12.
\end{multline}
  \item We now bound the term $\dps\inner{\bar \vu(0)}{\int_0^\tau \vQ(\tilde \vu) \vphi \,dt}$ by bounding its typical term $\dps \int_0^\tau\inner{\bar u(0)} {\tilde u^2 \phi^u}dt $:

\begin{multline*}
\inner{\ubrace{\bar u(0)}{\in L^\f{2d}{6-d} } } {\ubrace{\tilde u^2}{L^\f{d}{d-2}} \ubrace{\phi^u}{\in L^\f{2d}{d-2}}} \le \abs{\bar u(0)}_{L^\f{2d}{6-d}} \abs{\tilde u}_{L^\f{2d}{d-2}}^2 \abs{\phi^u}_{L^\f{2d}{d-2}} \\\le c \abs{\bar u(0)}_{L^\f{2d}{6-d}} \norm{\tilde u}_{L^\infty(0,T;H^1)}^2 \norm{\phi^u}_{L^\infty(0,T;H^1)}.
\end{multline*}
Therefore, by \eqref{phi bounda}, we find
\begin{equation}\label{4.8}
\inner{\bar \vu(0)}{\int_0^\tau \vQ(\tilde \vu) \vphi \,dt} \le cT \abs{\bar \vu(0)}_{L^\f{2d}{6-d}} \norm{\tilde \vu}_{L^\infty(0,T;H^1)}^2 \norm{\vchi}_{H^1}.
\end{equation}
  \item We finally bound
  \begin{equation} \label{4.9}
    \inner {\bar \vu(0)} {\int_0^\tau \vl(\vphi)dt} \le c \abs{\bar \vu(0)} \int_0^ \tau \abs{\vphi}dt \le c T^\f12 \abs{\bar \vu(0)} \norm{\vchi}^\f12.
  \end{equation}
\end{enumerate}

We therefore infer from \eqref{4.5}--\eqref{4.9} that
\[\inner{\bar\vu(\tau)}{\vchi} \le \inner{\bar{\vu}(0)}{\vchi} + \kappa\left( \left[\abs{\tilde \vu}_{L^\infty(0,T;H^1)}^\f12 +1\right]\norm{\vchi}_{H^1}^\f12+ \norm{\tilde \vu}_{L^\infty(0,T;H^1)}^2 \norm{\vchi}_{H^1} \right) \abs{\bar{\vu}(0)}_{L^q(\Omega)^2} , \]
where $\kappa = \kappa(T)$ is independent of $\tau$ and $q=\max(2d/(6-d),4d/(d+2))$.

By taking the supremum over $\vchi$ with $\norm{\vchi}= \norm{\chi}_{H^1(\Omega)^2} \le 1$, we conclude that
\begin{equation}
  \label{initial dep}
   \sup_{\vchi \in H^1: \norm{\vchi} \le 1}\inner{\vu_1(\tau) -\vu_2(\tau)}{\vchi} \le \abs{\vu_1(0) -\vu_2(0)} +\kappa \abs{\vu_1(0) -\vu_2(0)}_{L^q(\Omega)^2},
\end{equation}
where $\kappa =\kappa (T,\norm{\vu_1+ \vu_2}_{L^\infty(0,T;H^1)})$ .
\appendix
 \section*{Appendices}

\section{A technical lemma}
 \begin{Lem} \label{lem: u diff} Suppose that $\vp,\vq$ are as in \eqref{pi qi li}  and  $\vP,\vQ$ are as in \eqref{P Q jac}.
 We then have
 \begin{equation}
   \label{p diff}
  \vp(\vu_1) - \vp(\vu_2) = \vP(\tilde{\vu})\bar{\vu},
 \end{equation}
  and
 \begin{equation}
   \label{q diff}
  \vq(\vu_1) - \vq(\vu_2) = \vQ(\tilde{\vu})\bar{\vu},
 \end{equation}
 where $\tilde{\vu}={(\vu_1+\vu_2)}/{2}$  and $\bar{\vu}=\vu_1-\vu_2$.
 \end{Lem}
 \begin{proof}
   We write
   \begin{align*}
     &\vp(\vu_1) -\vp(\vu_2)  = \vp(\vu_2 + \bar\vu) -\vp(\vu_2) = \int_0^1 \f{d}{dt} \vp (\vu_2 + t\bar{\vu})\;dt
      =\int_0^1 \f{D \mc{P}}{D\vu}(\vu_2+t\bar{\vu}) \cdot \bar{\vu}\;dt
     \\& = \int_0^1 \begin{pmatrix}
       d_1 +2a_{11}(u_2+t\bar{u}) + a_{12}(v_2+t\bar{v}) & a_{12}(u_2+t\bar{u})
       \\ a_{21}(v_2+t\bar{v}) & d_2+a_{21}(u_2+t\bar{u})+2a_{22}(v_2+t\bar{v})
   \end{pmatrix}
   \cdot \bar{\vu} \;dt
   \\& = \begin{pmatrix}
     d_1 +a_{11}(u_1+u_2) + a_{12}{(v_1+v_2)}/{2} & a_{12}{(u_1+u_2)}/{2}
     \\ a_{21}{(v_1+v_2)}/{2} & d_2+a_{21}{(u_1+u_2)}/{2}+a_{22}(v_1+v_2)
 \end{pmatrix} \cdot \bar{\vu}
 \\& = \vP(\tilde\vu)\cdot \bar{\vu}.
   \end{align*}

   We thus proved \eqref{p diff} and  we can derive \eqref{q diff} in the same fashion.
 \end{proof}
\section{Existence result for SKT systems}

 In \cite[Theorem 3.1]{PT16}, we proved the following existence result for SKT system \eqref{SKT vector}:
 \begin{Thm}[Existence of solutions for the SKT]\label{thm: existence}{\color{white}s}
   \begin{enumerate}[ i) ]
     \item We assume that that $d\le 4$, that the condition \eqref{coef cond} hold, and that $\vu_0$ is given, $\vu_0\in L^2(\Omega)^2,\vu_0\ge 0$. Then equation \eqref{SKT vector} possesses a solution $\vu\ge \vo$ such that, for every $T>0$:
     \begin{subequations}
       \label{3.38}
\begin{align}
\label{3.36a}  &       \vu\in L^\infty(0,T;L^2(\Omega)) \cap L^2(0,T;H^1(\Omega)^2)\\
  &(\sqrt{u} +\sqrt{v})(\abs{\grad u}+ \abs{\grad v}) \in L^2(0,T;L^2(\Omega))\\
  \label{3.38c}&\vu \in L^4(0,T;L^4(\Omega)).
\end{align}
     \end{subequations}
     with the norms in these spaces bounded by a constant depending  on $T$, on the coefficients, and on the norms in $L^2(\Omega)$ of $u_0$ and $v_0$.
     \item If, in addition, $\grad \vp(\vu_0)\in L^2(\Omega)^4$, then the solution $\vu$ also satisfies
\begin{subequations}
  \begin{align}
    \label{3.36}
    &\grad \vp(\vu) \in L^\infty(0,T;L^2(\Omega)^4), \quad (1+\abs{u}+\abs{v})^\f12 \left(\abs{\partial_t u}+\abs{\partial_t v}\right) \in L^2(0,T;L^2(\Omega)),
  \\ &   \label{3.37}\Delta \vp(\vu) \in L^2(0,T;L^2(\Omega)^2),
\end{align}
\end{subequations}
      with the norms in these spaces bounded by a constant depending on the norms of $\vu_0$ and $\grad \vp(\vu_0)$ in $L^2$ (and on $T$ and the coefficients).
   \end{enumerate}
 \end{Thm}
 \section{Additional regularity of weak solutions\label{appen: C}}
Although this was not explicitly stated in \cite{PT16}, the solutions that we constructed in dimension $d\le 4$ belong to $L^\infty_t(H^1)$ with $\partial_t \vu$ in $L^2_t(L^2)$:
 \begin{enumerate}[\hoathib\;]
   \item From \eqref{3.36a} and \eqref{3.36}, we have $\vu,\grad \vp(\vu)\in L^\infty(0,T;L^2(\Omega)^2)$. To show that $\grad \vu \in L^\infty(0,T;L^2(\Omega)^2)$, we note that \eqref{P positive definite} implies that, for $u,v \ge 0$, $\vP(\vu)$ is invertible (as a $2\times 2$ matrix),
   and that, pointwise (i.e. for a.e. $x\in \Omega$),
   \begin{equation}\label{2.5b}
    \abs{\vP(\vu)^{-1}}_{\mathcal{L}(\mathbb{R}^2)} \le \f{1}{d_0+\alpha(u+v)}.
   \end{equation}
   We thus find $\grad \vu \in L^\infty(0,T;L^2(\Omega)^2)$ which says that $\vu\in L^\infty(0,T;H^1(\Omega)^2)$.
     \item From \eqref{3.36}, we have $\partial_t\vu \in L^2(\Omega_T)^2$.
 \end{enumerate}

 \paragraph{Acknowledgement.} This work was supported in part by NSF grant DMS151024 and by the Research Fund of Indiana University.
{\sffamily
\def\cfudot#1{\ifmmode\setbox7\hbox{$\accent"5E#1$}\else
  \setbox7\hbox{\accent"5E#1}\penalty 10000\relax\fi\raise 1\ht7
  \hbox{\raise.1ex\hbox to 1\wd7{\hss.\hss}}\penalty 10000 \hskip-1\wd7\penalty
  10000\box7} \def\cprime{$'$}

}


\begin{thebibliography}{ACP82}

\bibitem[ACP82]{ACP82}
Donald Aronson, Michael~G. Crandall, and L.~A. Peletier.
\newblock Stabilization of solutions of a degenerate nonlinear diffusion
  problem.
\newblock {\em Nonlinear Anal.}, 6(10):1001--1022, 1982.

\bibitem[Ama89]{Ama89}
Herbert Amann.
\newblock Dynamic theory of quasilinear parabolic systems. {III}. {G}lobal
  existence.
\newblock {\em Math. Z.}, 202(2):219--250, 1989.

\bibitem[Ama90]{Ama90}
Herbert Amann.
\newblock Dynamic theory of quasilinear parabolic equations. {II}.
  {R}eaction-diffusion systems.
\newblock {\em Differential Integral Equations}, 3(1):13--75, 1990.

\bibitem[Aro86]{Aro85}
D.~G. Aronson.
\newblock The porous medium equation.
\newblock In {\em Nonlinear diffusion problems ({M}ontecatini {T}erme, 1985)},
  volume 1224 of {\em Lecture Notes in Math.}, pages 1--46. Springer, Berlin,
  1986.

\bibitem[CDJ16]{CDJ16}
Xiuqing Chen, Esther~S. Daus, and Ansgar J{\"u}ngel.
\newblock Global existence analysis of cross-diffusion population systems for
  multiple species.
\newblock submitted--2016.

\bibitem[CJ04]{CJ04}
Li~Chen and Ansgar J{\"u}ngel.
\newblock Analysis of a multidimensional parabolic population model with strong
  cross-diffusion.
\newblock {\em SIAM J. Math. Anal.}, 36(1):301--322 (electronic), 2004.

\bibitem[J{\"u}n15]{Jun15}
Ansgar J{\"u}ngel.
\newblock The boundedness-by-entropy method for cross-diffusion systems.
\newblock {\em Nonlinearity}, 28(6):1963--2001, 2015.

\bibitem[LM72]{LM72}
J.-L. Lions and E.~Magenes.
\newblock {\em Non-homogeneous boundary value problems and applications. {V}ol.
  {I}}.
\newblock Springer-Verlag, New York-Heidelberg, 1972.
\newblock Translated from the French by P. Kenneth, Die Grundlehren der
  mathematischen Wissenschaften, Band 181.

\bibitem[LU68]{Lad68}
Olga~A. Ladyzhenskaya and Nina~N. Ural{\cprime}tseva.
\newblock {\em Linear and quasilinear elliptic equations}.
\newblock Translated from the Russian by Scripta Technica, Inc. Translation
  editor: Leon Ehrenpreis. Academic Press, New York-London, 1968.

\bibitem[PT17]{PT16}
D.~Pham and R.~Temam.
\newblock Weak solutions of the shigesada-kawasaki-teramoto equations and their
  attractors.
\newblock {\em Nonlinear Anal.}, in press -- 2017.

\bibitem[SKT79]{SKT79}
Nanako Shigesada, Kohkichi Kawasaki, and Ei~Teramoto.
\newblock Spatial segregation of interacting species.
\newblock {\em J. Theoret. Biol.}, 79(1):83--99, 1979.

\bibitem[Yag93]{Yag93}
Atsushi Yagi.
\newblock Global solution to some quasilinear parabolic system in population
  dynamics.
\newblock {\em Nonlinear Anal.}, 21(8):603--630, 1993.

\bibitem[Yag08]{Yag08}
Atsushi Yagi.
\newblock Exponential attractors for competing species model with
  cross-diffusions.
\newblock {\em Discrete Contin. Dyn. Syst.}, 22(4):1091--1120, 2008.

\end{thebibliography}
\end{document}